\documentclass[10pt]{amsart}

\usepackage[utf8]{inputenc}

\usepackage{url}

\usepackage{arydshln}

\usepackage{enumitem}

\usepackage{rotating}

\usepackage[small, width=0.9\textwidth]{caption}

\usepackage{tikz-cd}

\usepackage{amssymb}
\usepackage{amsthm}

\usepackage{hyperref}

\usepackage{changes}

\usepackage{mathrsfs}

\numberwithin{equation}{section}

\theoremstyle{plain}
\newtheorem{thm}[equation]{Theorem}
\newtheorem{lem}[equation]{Lemma}
\newtheorem{cor}[equation]{Corollary}
\newtheorem{prop}[equation]{Proposition}

\newtheorem{rem}[equation]{Remark}

\newtheorem*{thm*}{Theorem}

\theoremstyle{definition}
\newtheorem{defn}[equation]{Definition}

\newtheorem{ex}[equation]{Example}

\newtheorem*{ex*}{Example}

\let\tilde = \widetilde

\newcommand{\bbC}{{\mathbb C}}

\newcommand{\bbN}{{\mathbb N}}
\newcommand{\bbP}{{\mathbb P}}

\newcommand{\bbZ}{{\mathbb Z}}

\newcommand{\scrA}{{\mathscr A}}

\newcommand{\scrD}{{\mathscr D}}
\newcommand{\scrE}{{\mathscr E}}

\newcommand{\scrG}{{\mathscr G}}

\newcommand{\scrL}{{\mathscr L}}

\newcommand{\scrN}{{\mathscr N}}
\newcommand{\scrO}{{\mathscr O}}

\newcommand{\scrV}{{\mathscr V}}

\newcommand{\scrX}{{\mathscr X}}

\newcommand{\scrZ}{{\mathscr Z}}

\newcommand{\codim}{\mathrm{codim}}

\newcommand{\Supp}{\mathrm{Supp}}

\newcommand{\Sing}{\mathrm{Sing}}
\newcommand{\Sm}{\mathrm{Sm}}
\newcommand{\CH}{\mathrm{CH}}

\renewcommand{\sp}{{\mathrm{sp}}}

\newcommand{\Sym}{\mathrm{Sym}}

\newcommand{\Perv}{\mathrm{Perv}}
\newcommand{\IC}{{\mathrm{IC}}}

\newcommand{\Pic}{{\mathit{Pic}}}

\newcommand{\Spec}{{\mathrm{Spec}}}
\newcommand{\Proj}{{\mathrm{Proj}}}

\renewcommand{\P}{\mathbb{P}}

\newcommand{\scrHom}{{\mathscr{H}\kern -.9pt om}}
\newcommand{\scrExt}{{\mathscr{E}\kern -.9pt xt}}

\DeclareMathOperator{\CC}{CC}

\DeclareMathOperator{\Gr}{Gr}
\DeclareMathOperator{\Fl}{Fl}

\makeatletter
\def\@seccntformat#1{%
  \protect\textup{\protect\@secnumfont
    \ifnum\pdfstrcmp{subsection}{#1}=0 \bfseries\fi
    \csname the#1\endcsname
    \protect\@secnumpunct
  }%
}  
\makeatother

\makeatletter
\@namedef{subjclassname@2020}{%
  \textup{2020} Mathematics Subject Classification}
\makeatother

\begin{document}

\vspace*{-0.5em}

\title[Semicontinuity of Gauss maps]{Semicontinuity of Gauss maps\\ and the Schottky problem}

\author{Giulio Codogni}
\address{Dipartimento di Matematica, Università degli Studi Tor Vergata \newline \hspace*{1em} Via della ricerca scientifica 1,	00133 Roma (Italy)}
\email{codogni@mat.uniroma2.it}

\author{Thomas Kr\"amer}
\address{Institut f\"ur Mathematik, Humboldt-Universit\"at zu Berlin \newline \hspace*{1em} Unter den Linden 6, 10099 Berlin (Germany)}
\email{thomas.kraemer@math.hu-berlin.de}

\keywords{Gauss map, abelian variety, theta divisor, Schottky problem.}
\subjclass[2020]{Primary 14K12; Secondary 14C17, 14F10, 14H42}

\begin{abstract}
We show that the degree of Gauss maps on abelian varieties is semicontinuous in families, and we study its jump loci. As an application we obtain that in the case of theta divisors this degree answers the Schottky problem. Our proof computes the degree of Gauss maps by specialization of Lagrangian cycles on the cotangent bundle. We also get similar results for the intersection cohomology of varieties with a finite morphism to an abelian variety; it follows that many components of Andreotti-Mayer loci, including the Schottky locus, are part of the stratification of the moduli space of ppav's defined by the topological type of the theta divisor.
\end{abstract}

\maketitle
\setcounter{tocdepth}{1}
\vspace*{-0.8em}
\tableofcontents
\vspace*{-0.8em}

\thispagestyle{empty}

\section{Introduction}

The Gauss map of a hypersurface in projective space is the rational map that sends any smooth point of the hypersurface to its normal direction in the dual projective space. The analogous notion of Gauss maps for subvarieties of abelian varieties appears already in Andreotti’s proof of the Torelli theorem~\cite{Andreotti}. In contrast to the case of projective hypersurfaces, the Gauss map for any ample divisor on an abelian variety is generically finite of degree $>1$, and its degree is related to the singularities of the divisor. We show that this degree is lower semicontinuous in families, and we study its jump loci. As an application we get that in the moduli space of principally polarized abelian varieties, the degree of the Gauss map refines the Andreotti-Mayer stratification and answers the Schottky problem as conjectured in~\cite{CGS}. We work over an algebraically closed field $k$ with $\mathrm{char}(k)=0$. In section~\ref{sec:IC} we obtain similar results for the intersection cohomology of complex varieties with a finite morphism to an abelian variety. In particular, many Andreotti-Mayer loci such as the Schottky locus are determined over the complex numbers already by the topological type of the theta divisor.  


\subsection{Gauss maps and their jump loci}

Let $A$ be an abelian variety over $k$. By translations we may identify its tangent spaces at all points, hence the cotangent bundle $T^\vee A = A\times V$ is trivial with fiber $V=H^0(A, \Omega^1_A)$. The {\em Gauss map} of a reduced effective divisor $D\subset A$ is the rational map
\[
 \gamma_D: \quad D \;\dashrightarrow\; \P V
\] 
that sends a smooth point of the divisor to its conormal direction at that point; it coincides with the rational map given by the linear series $\P V^\vee = |\scrO_D(D)|$.~For an irreducible divisor this is a generically finite dominant map iff the divisor is ample, which happens iff the divisor is not stable under translations by any positive dimensional abelian subvariety~\cite[cor.~II.11, lem.~II.9]{RanSubvarieties}. Even in the generically finite case the Gauss map can have positive dimensional fibers~\cite{AC}.

\medskip 

For algebraic families of generically finite maps the generic degree always defines a constructible stratification of the parameter space, but in general it can jump in both directions (see example~\ref{ex:not-semicontinuous}). Our first semicontinuity result says that for Gauss maps on abelian varieties this does not happen:

\begin{thm} \label{thm:semicontinuous}
Let $A \to S$ be an abelian scheme over a variety $S$, and let $D\subset A$ be a relatively ample divisor which is flat over $S$. Let $D_s\subset A_s$ denote their fibers over $s\in S$, and let $\gamma_{D_s}$ be the corresponding Gauss map. Then for each $d\in \bbN$ the subsets
\[
 S_d \;=\; \bigl\{ s\in S \mid \deg(\gamma_{D_s}) \leq d \bigr\} \;\subseteq \; S
\] 
are closed in the Zariski topology.
\end{thm} 

The above result does not show where the degree actually jumps. Let us say that an irreducible subvariety of an abelian variety is {\em negligible} if it is stable under translations by a positive dimensional abelian subvariety. Simple abelian varieties have no negligible subvarieties other than themselves. More generally, by~\cite[th.~3]{Abramovich} 
an irreducible closed subvariety of an abelian variety is negligible iff it is not of general type. Our second result says that in the setting of theorem~\ref{thm:semicontinuous} the degree of the Gauss map jumps whenever a new component of general type appears in the singular locus $\Sing(D_s)$. To make this precise we specify a curve along which we move inside the parameter space:

\begin{thm} \label{thm:jump}
Let $S'\subset S$ be a curve, and fix a point $0\in S'(k)$. If $\Sing(D_0)$ has an irreducible component which is of general type and not contained in the Zariski closure of~$\bigcup_{s\in S'\setminus \{0\}} \Sing(D_s)$, then 
\[ 
 \deg(\gamma_{D_0}) < \deg(\gamma_{D_s})
\]
for all $s\in S'(k) \setminus \{0\}$ in some Zariski open neighborhood of the point $0$.
\end{thm}

The above in particular applies if all components of the singular locus are of general type and~$\dim(\Sing(D_0))>\dim(\Sing(D_{s}))$ for all~$s\neq 0$. This last condition is motivated by the case of theta divisors and the Schottky problem.


\subsection{Application to the Schottky problem} Let $\scrA_g$ be the moduli space of principally polarized abelian varieties of dimension $g$. Inside it, consider for $d\in \bbN$ the Gauss loci
\[
 \scrG_d \;=\; \{ (A, \Theta) \in \scrA_g \mid \deg(\gamma_\Theta) \leq d\} 
\] 
as in~\cite[sect.~4]{CGS} The above results show that these loci are closed (cor.~\ref{cor:closed}) and refine the Andreotti-Mayer stratification (cor.~\ref{cor:AM}). 
Thus the Gauss loci provide a solution for the Schottky problem to characterize the closure of the locus of Jacobians in the moduli space of principally polarized abelian varieties:

\begin{cor} \label{cor:Schottky_intro}
Inside $\scrA_g$ we have: \medskip 

\noindent (a) The locus of Jacobians is a component of $\scrG_{d}$ for $d=\binom{2g-2}{g-1}$.\smallskip

\noindent (b) The locus of hyperelliptic Jacobians is a is a component of $\scrG_{d}$ for $d=2^{g-1}$.\smallskip
\end{cor}

The above corollary is shown in section~\ref{sec:schottky} together with an analogous statement for Prym varieties. It confirms a conjecture by the first author, Grushevsky and Sernesi~\cite[conjecture~1.6]{CGS} who verified it for $g\leq 4$ by an explicit description of the Gauss loci. As pointed out in loc.~cit., this is only a weak solution to the Schottky problem: In general the Gauss loci in the above corollary have more than one irreducible component and the Jacobian locus is only one of them. The theory of $\scrD$-modules allows to refine the degree of the Gauss map to representation theoretic invariants that might distinguish the Jacobian locus~\cite{KraemerMicrolocalII}.


\subsection{The degree of conormal varieties}

For the proof of theorem~\ref{thm:semicontinuous}~and~\ref{thm:jump} we interpret the degree of the Gauss map as an intersection number of Lagrangian cycles on the cotangent bundle of the abelian variety and apply specialization for such cycles~\cite{GinzburgCharacteristic, SabbahConormaux}, which we can do because $\mathrm{char}(k)=0$. To explain how this works, let us forget about abelian varieties for a moment and fix any ambient smooth variety $W$ over $k$. The {\em conormal variety} to a subvariety $X\subset W$ is defined as the closure \medskip 
\[
 \Lambda_X \;=\; \overline{ \strut \{ (x,\xi) \mid x\in \Sm(X), \;\xi \in T_x^\vee(W), \; \xi \perp T_x (X) \}} \;\subset\; T^\vee(W) \medskip
\]
of the conormal bundle to the smooth locus $\Sm(X)$, where the closure is taken in the total space of the cotangent bundle of the ambient smooth variety. This conormal variety always has pure dimension $n=\dim(W)$, in fact it is Lagrangian with respect to the natural symplectic structure on the cotangent bundle. It is also conic, i.e.~stable under the natural action of the multiplicative group on the fibers of the cotangent bundle. Conversely, any closed conic Lagrangian subvariety of the cotangent bundle arises like this~\cite[lemma~3]{Kennedy}. So the map $X\mapsto \Lambda_X$ induces an isomorphism \medskip 
\[
 \scrZ(W) = \{ \textnormal{cycles on $W$}\} 
 \;\stackrel{\sim}{\longrightarrow}\; \scrL(W) = \{ \textnormal{conic Lagrangian cycles on $T^\vee W$}\},  \medskip 
\]
where by a {\em conic Lagrangian cycle} we mean a $\bbZ$-linear combination of closed conic Lagrangian subvarieties. In the case of projective varieties we can talk about the degree of conormal varieties:

\begin{defn} \label{defn:degree}
If $W$ is projective, the {\em degree homomorphism} on conic Lagrangian cycles is the map
\[
 \deg: \quad \scrL(W) \;\longrightarrow\; \CH^n(T^\vee(W)) \;\stackrel{i^*}{\longrightarrow}\; \CH^n(W) \;\twoheadrightarrow \; \bbZ
\]
which is given by the intersection number with the zero section $i: W \hookrightarrow T^\vee(W)$.
\end{defn}

\begin{ex} \label{ex:local-euler}
Over the complex numbers the above degree can be computed as follows. For any constructible function $F: W\to \bbZ$ consider the topological Euler characteristic 
\[
 \chi_\mathrm{top}(W, F) \;=\; \sum_{n\in \bbZ} \, n\cdot \chi_\mathrm{top}(F^{-1}(n)),
\]
where $\chi_\mathrm{top}(F^{-1}(n))$ denotes the alternating sum of the Betti numbers of $F^{-1}(n)$.~In the theory of characteristic classes of singular varieties this definition is applied to a particular constructible function $\mathrm{Eu}_X: W \to \bbZ$, the {\em local Euler obstruction} of a subvariety $X\subseteq W$~\cite[sect.~3]{MacPhersonChern}. Outside of the singular locus $\Sing(X)\subseteq X$ it has the form
\[
 \mathrm{Eu}_X(p) \;=\;
 \begin{cases} 
 0 & \textnormal{for $p\in W\setminus X$}, \\
 1 & \textnormal{for $p\in \Sm(X)$},
 \end{cases} 
\]
but its values on the singular locus depend on the singularities: For instance,  if~$X$ is a curve, $\mathrm{Eu}_X(p)$ is the multiplicity of that curve at $p$. In general, the degree of conormal varieties in definition~\ref{defn:degree} can be expressed as an Euler characteristic by the formula
\[
 \deg(\Lambda_X) \;=\; (-1)^{\dim(X)} \cdot \chi_\mathrm{top}(X, \mathrm{Eu}_X), \medskip
\] 
see~\cite[lemme 1.2.1]{SabbahConormaux}~\cite[prop.~6.1(b)]{Fulton}. The right hand side can be computed easily as soon as we know the local Euler obstruction. For a smooth rational curve $X$ we get $\deg(\Lambda_X)=-2$ while a nodal or cuspidal cubic has $\deg(\Lambda_X)=-3$. Note that a cuspidal cubic is homeomorphic to a smooth rational curve, so the degree of conormal varieties is not a topological invariant. Moreover, it can be negative.  
\end{ex}


\subsection{Proof of the semicontinuity theorems}

Of course there are no rational curves in abelian varieties, and in the case of abelian varieties the degree behaves much better. By~\cite[th.~1]{WeissauerDegenerate} we have the following result (see section~\ref{sec:abvar}):

\begin{prop} \label{prop:degree-positive}
If $W=A$ is an abelian variety, then \smallskip
\begin{itemize} 
\item $\deg(\Lambda_X)\geq 0$ for any $X\subset A$, \smallskip 
\item $\deg(\Lambda_X)>0$ if and only if $X$ is of general type, \smallskip
\item $\deg(\Lambda_X)=\deg(\gamma_X)$ for divisors $X\subset A$ with Gauss map $\gamma_X$. \smallskip
\end{itemize} 
\end{prop}

This easily implies theorem~\ref{thm:semicontinuous} when combined with the principle of Lagrangian specialization which we recall in section~\ref{sec:specialization}: For any flat family of subvarieties in a smooth ambient 1-parameter family, the limit of their conormal varieties is an effective conic Lagrangian cycle whose support contains the conormal variety to the central fiber as a component, and the total degree of the limit cycle equals the degree of a general fiber. The same argument shows that our semicontinuity result holds not only for divisors but for subvarieties of any codimension: 

\begin{thm} \label{thm:semicontinuous-in-any-codim}
Let $A \to S$ be an abelian scheme over a variety $S$, and let $X\subset A$ be an arbitrary family of subvarieties which is flat over $S$. Then for each $d\in \bbN$ the subsets
$
 S_d = \bigl\{ s\in S \mid \deg(\Lambda_{X_s}) \leq d \bigr\} \subseteq  S
$ 
are closed in the Zariski topology.
\end{thm}

It remains to prove theorem~\ref{thm:jump}. Given the interpretation for the degree of Gauss maps in prop.~\ref{prop:degree-positive}, the proof has nothing to do with abelian varieties: In section~\ref{sec:jump} we show that for any flat family of divisors on a smooth 1-parameter variety, the specialization of their conormal varieties contains an extra component whenever the singular locus of the fiber jumps. While the final criterion is phrased only for divisors, we formulate our arguments as far as possible for subvarieties in arbitrary codimension to get beyond theorem~\ref{thm:jump} (see example~\ref{ex:higher-codim}). This is important even if one only wants to study singularities of divisors: In the theory of Chern classes for singular varieties one attaches to any subvariety $X\subset A$ a characteristic cycle of the form
\[ \Lambda \;=\; \Lambda_X + \sum_Z m_Z \Lambda_Z 
\] 
where $Z$ runs through certain strata in $\Sing(X)$~\cite{Kennedy, SabbahConormaux}, and the topologically meaningful invariant that appears in generalizations of the Gauss-Bonnet index formula is the total degree $\deg(\Lambda)$ involving all the strata.


\subsection{A topological view on jump loci}

In section~\ref{sec:IC}, which is not used in the rest of the paper, we deduce from our previous results a general semicontinuity theorem for the intersection cohomology of varieties over the complex numbers. Recall that for a complex variety $X$, the intersection cohomology $\mathrm{IH}^\bullet(X)$ only depends on its homeomorphism type in the Euclidean topology; it coincides with Betti cohomology in the smooth case but is better behaved in general~\cite{BorelIC,GMIntersectionI, GMIntersectionII, KirwanWoolf, MaximIHP}. 
We denote by
\[
 \chi_\mathrm{IC}(X) \;=\; \sum_{i\geq 0} \; (-1)^{i+\dim(X)} \dim \mathrm{IH}^i(X)
\]
the Euler characteristic of the intersection cohomology. This Euler characteristic is usually not semicontinuous in families, it can jump in both directions. But for families of finite branched covers of subvarieties in complex abelian varieties this does not happen (see lemma~\ref{lem:signed-Euler} and corollary~\ref{cor:IC-semicontinuous}):

\begin{thm}
Let $f: X \to S$ be a family of varieties such that each fiber $X_s$ is generically reduced and admits a finite morphism to an abelian variety. Then for each $d\in \bbN_0$ the loci
\[
 S_d \;:=\; \{ s\in S \mid \chi_\IC(X_s) \leq d \} \;\subseteq\; S
\]
are closed in the Zariski topology.
\end{thm}

This puts our results in a topological context, since the intersection cohomology of a complex variety only depends on its homeomorphism type. For instance, it follows from the above that a singular theta divisor cannot be  homeomorphic to a smooth one (recall that there are examples of normal varieties which are singular but homemorphic to smooth varieties, such as those by Brieskorn~\cite{BrieskornDiffTop, BrieskornSingularNormal}). In corollary~\ref{cor:schottky-topology} we will see that the Jacobian locus appears in the stratification of $\scrA_g$ by the intersection cohomology of the theta divisor, so we obtain:

\begin{cor} \label{cor:intro-schottky-topology}
The locus of Jacobian varieties in $\scrA_g$ is an irreducible component of the closure of the locus of all ppav's whose theta divisor is homeomorphic to a theta divisor on a Jacobian variety.
\end{cor} 

It seems an interesting problem to study the topology of theta divisors on abelian varieties in more detail. 


\subsection*{Acknowledgements}

The authors would like to thank Sam Grushevsky, Ariyan Javanpeykar, Constantin Podelski, Claude Sabbah, Edoardo Sernesi and the referee for helpful comments. G.C. is funded by the MIUR {\it Excellence Department Project}, awarded to the Department of Mathematics, University of Rome, Tor Vergata, CUP E83C18000100006, and PRIN 2017 {\it Advances in Moduli Theory and Birational Classification}. T.K. is supported by DFG Research Grant 430165651 {\it Characteristic Cycles and Representation Theory}.
\medskip

\section{Lagrangian specialization} \label{sec:specialization}

For convenience we include in this section a self-contained review of some basic facts about the specialization of Lagrangian cycles, which was introduced in relation with Chern-MacPherson classes~\cite{SabbahConormaux} and  nearby cycles for $\scrD$-modules and perverse sheaves~\cite{GinzburgCharacteristic}. We work in a relative setting over a smooth curve $S$. The family of our ambient spaces is given by a smooth dominant morphism of varieties $f:W\to S$ where $\dim(W)=n+1$. Let $X\subset W$ be a reduced closed subvariety. The relative smooth locus \medskip
\[
 \Sm(X/S) \;=\; \{ \, x\in \Sm(X) \mid \; \textnormal{the restriction $f|_X: X\to S$ is smooth at $x$} \,\}
\]
is nonempty iff $\dim f(X)>0$, in which case $X\to S$ is flat and $\Sm(X/S)\subset X$ is an open dense subset because $\mathrm{char}(k)=0$. Any $x\in \Sm(X/S)$ is a smooth point of the fiber $X_{s} = f^{-1}(s)\cap X$ over the image point $s=f(x)$. Hence inside the total space
\[
 T^\vee(W/S) \;=\; T^\vee(W)/f^{-1}T^\vee(S)
\]
of the relative cotangent bundle, we define the {\em relative conormal variety} to $X$ as the closure \medskip 
\[
 \Lambda_{X/S} = \overline{\strut \{ (x,\xi) \in T^\vee(W/S) \mid x\in \Sm(X/S), \; \xi \perp T_x X_{f(x)} \}} \;\subset\; T^\vee(W/S). \medskip 
\]

\begin{rem} \label{rem:bmm}
In~\cite{BMM} the relative conormal variety is instead defined as the closure inside the absolute cotangent bundle. This notion of relative conormal variety is obtained from ours by base change via the quotient map $T^\vee(W) \twoheadrightarrow T^\vee(W/S)$, i.e.~we have
 \[
 T^\vee(W) \times_{T^\vee(W/S)} \Lambda_{X/S} 
 \;=\; 
 \overline{ \strut \{ (x,\xi) \in T^\vee(W) \mid x\in \Sm(X/S), \; \xi \perp T_x X_{f(x)} \}}.
 \]
\end{rem}

Indeed, both sides are irreducible closed subvarieties of $T^\vee( W)|_X$. For the right hand side this holds by definition, for the left hand side it follows from the fact that $\Lambda_{X/S}\subset T^\vee(X/S)$ is an irreducible closed subvariety and $T^\vee W\to T^\vee(X/S)$ is a fibration with irreducible fibers. So it suffices to show that both sides agree over some open dense $U \subset X$. We can assume $X$ is flat over $S$ and take $U=\Sm(X/S)$, in which case the claim becomes obvious.

\begin{lem} \label{rem:flat}
If $X$ is flat and irreducible over $S$, then so is $\Lambda_{X/S}$.
\end{lem}

\begin{proof}
$\Lambda_{X/S}$ is defined as the schematic closure of a locally closed subscheme $V$ of the relative cotangent bundle $T^\vee(A/S)$. The subscheme $V$ is the total space of a vector bundle over a smooth variety, so it is a smooth variety as well. Its schematic closure is integral, and a morphism from an integral scheme to a smooth curve is flat iff it is dominant~\cite[chapter III, prop.~9.7]{Hartshorne}.
\end{proof}

\medskip

Relative conormal varieties can be seen as families of conormal varieties. In what follows we denote by
\[
 \scrL(W/S) \;=\; \bigoplus_{X\subset W} \bbZ \cdot \Lambda_{X/S}
\] 
the free abelian group on relative conormal varieties to closed subvarieties $X\subset W$ that are flat over $S$. By the \emph{specialization} of $\Lambda \in \scrL(W/S)$ at~$s\in S(k)$ we mean the cycle
\[
 \sp_s(\Lambda) \;=\; \bigl[ \Lambda \cdot f^{-1}(s) \bigr] 
\]
which underlies the schematic fiber of the morphism $\Lambda_{X/S} \to S$ at~$s$. This is again a conic Lagrangian cycle by the following classical result, see~\cite[prop.~(a), p.~179]{FKM} or in an analytic setup \cite[sect.~1.2]{Teissier-Trang}:

\begin{lem}[{{\bf{Principle of Lagrangian specialization}}}] \label{lem:specialization}
The specialization at $s$ gives a homomorphism 
$$ \sp_s: \quad \scrL(W/S) \;\longrightarrow\; \scrL(W_s) \medskip $$ 
sending effective cycles to effective cycles. On Chow groups it induces the Gysin map in the bottom row of the following commutative diagram:\medskip 
\[
\begin{tikzcd} 
\scrL(W/S) \ar[r, "\sp_s"] \ar[d] & \scrL(W_s) \ar[d] \\
\CH^n(T^\vee(W/S)) \ar[r, "i_s^*"] & \CH^{n-1}(T^\vee(W_s)) 
\end{tikzcd} \medskip
\]
For any closed subvariety $X\subset W$ which is flat over $S$, there is a finite subset $\Sigma \subset S$ such that
\[ 
 \sp_s(\Lambda_{X/S}) 
 \;=\;
 \begin{cases} 
 \Lambda_{X_s} & \textnormal{\em for $s\in S\setminus \Sigma$}, \\
 m_{X_s}\Lambda_{X_s} + \sum_{Z\subset \Sing(X_s)} m_{Z}\Lambda_Z & \textnormal{\em for $s\in \Sigma$},
 \end{cases}   \medskip 
\]
where $m_{X_s}, m_{Z} > 0$ and the sum runs over finitely many subvarieties $Z\subset \Sing(X_s)$.
\end{lem}

\begin{proof}
Note that $T^\vee(W_s)$ is an effective Cartier divisor in $T^\vee(W)$. It intersects properly any relative conormal variety to a subvariety which is flat over $S$. Hence it is clear that the specialization induces on Chow groups the Gysin map defined in~\cite[sect.~2.6]{Fulton} and sends effective cycles to effective cycles. 

\medskip 

Now take an irreducible subvariety $X\subset W$ which is flat over $S$. By Lemma~\ref{rem:flat} the morphism $\Lambda_{X/S} \to S$ is flat and hence all its fibers are pure dimensional of the same dimension. Furthermore the action of the multiplicative group preserves the fibers of~$T^{\vee}(W/S)\to S$ and so the fibers of $\Lambda_{X/S}\to S$ are unions of conic subvarieties. As the canonical relative symplectic form on $T^{\vee}(W/S)$ restricts to the canonical symplectic form on $T^{\vee}(W_s)$ for every $s$, we conclude that the fibers of $\Lambda_{X/S}\to S$ are also Lagrangian and hence a union of conormal varieties, since the conic Lagrangian subvarieties of the cotangent bundle are precisely the conormal varieties~\cite[lemma~3]{Kennedy}. The coefficients are non-negative as the specialization of effective cycles is effective. Hence $\sp_s(\Lambda_{X/S})$ is a sum of conormal varieties, and since under the morphism $T^\vee(A_s) \to A_s$ its support surjects onto $X_s$, we conclude that one of the appearing components must be $\Lambda_{X_s}$.

\medskip 

As $\Lambda_{X/S}$ is irreducible and we work over a field of characteristic zero, there exists a Zariski open dense subset of $S$ over which the fibers of the morphism $\Lambda_{X/S}\to S$ are reduced and irreducible. We conclude that for $s$ in this Zariski open dense subset of $S$ we have $\sp_s(\Lambda_{X/S})=\Lambda_{X_s}$. Moreover, the specialization cannot have any further components over the relative smooth locus $\Sm(X/S)\subseteq X$, since on that locus also the morphism $\Lambda_{X/S}\to S$ restricts to a smooth morphism.
\end{proof}

\medskip

We have the following consequence of flatness:

\begin{prop} \label{prop:degree-constant}
Let $X\subset W$ be a closed subvariety which is flat over $S$, then the degree 
\[ d = \deg(\sp_s(\Lambda_{X/S})) \quad \textnormal{\em is independent of $s\in S(k)$}. \]
\end{prop}

\begin{proof}
The cycle class of the specialization $\sp_s(\Lambda_{X/S})$ is the image of $[\Lambda_{X/S}]$ under the Gysin map in lemma~\ref{lem:specialization}, and its degree is defined as the intersection number of this image with the zero section 
$
 X=\mathcal{X} \hookrightarrow \mathcal{W} = T^\vee(W/S)
$. As $\mathcal{X}$ and~$\mathcal{W}$ are flat over $S$ and $\mathcal{X} \hookrightarrow \mathcal{W}$ is a regular embedding, the degree is therefore constant by \cite[th.~10.2]{Fulton} applied to the relative conormal variety $\mathcal{V}=\Lambda_{X/S}$. Note that $\mathcal{V}\hookrightarrow \mathcal{W}$ is not required to be a regular embedding; in order to apply loc.~cit.~we only need that its base change to $\mathcal{X}$ is proper over $S$, which is true. 
\end{proof}

\begin{rem}
For $\mathrm{char}(k)>0$ the specialization of a family of conormal varieties need not be a sum of conormal varieties, see~\cite[p.~215]{KleimanTangencyAndDuality}. Similarly, for $\mathrm{char}(k)=0$ we need $\dim(S)=1$, otherwise we would have to restrict the class of morphisms as in~\cite{SabbahSansEclatement}. For instance, for $S=\bbC^2\ni s=(0,0)$ we have $i_s^{-1}(\Lambda_{X/S}) = T^\vee(W/S)|_{X_s}$ for the subvariety 
$
 X =  \{ ((x,y,z), (x^2 - y^2 z, y)) \mid (x,y,z)\in \bbC^3 \} \subset W = \bbC^3 \times S.
$
\end{rem}

\medskip

\section{Jump loci for the degree} \label{sec:jump}

Let $f: W\to S$ be a smooth dominant morphism from a smooth variety to a smooth curve as above. For any $S$-flat subvariety $X\subset W$ and $s\in S(k)$ we have seen that 
$
 \sp_s(\Lambda_{X/S}) - \Lambda_{X_s} \geq 0,
$ 
where the inequality means that the cycle on the left hand side is effective or zero. It is natural to ask for which $s\in S(k)$ the above inequality is strict. In the notation of lemma~\ref{lem:specialization} this happens iff $m_Z > 0$ for some~$Z\subseteq \Sing(X_s)$. The following provides a sufficient criterion for this to happen for families of divisors:

\begin{prop} \label{prop:component-in-sp}
Assume $d=\codim(X, W)=1$. If $\Sing(X/S)$ has an irreducible component which is contained in the fiber over some point $s\in S(k)$, then for this component $Z\subset X_s$ we have
$$
\sp_s(\Lambda_{X/S})-\Lambda_{X_s} \; \geq \; \Lambda_Z.
$$
\end{prop}

We divide the proof in several steps. Most of the argument works in arbitrary codimension $d$, so for the moment we do not yet assume $d=1$. It will be enough to prove the claim over some open dense subset of $W$. Fixing a general point $p\in Z(k)$ and working locally near that point, we can assume that \medskip 
\begin{itemize} 
\item $Z=\Sing(X/S)$ (equality as a scheme), \medskip 
\item $T^\vee(W/S) \simeq W\times V$ is the trivial bundle with fiber $V=T^\vee_p(W_s)$, \medskip 
\item $X\subset W$ is cut out by a regular sequence $f_1, \dots, f_d \in H^0(W,\mathcal{O}_W)$. \medskip 
\end{itemize} 
By the first item each $x\in (X\setminus Z)(k)$ is a smooth point of $X_t$ for $t=f(x)$. Fixing a trivialization as in the second item, we can furthermore identify the conormal space to $X_t\subset W_t$ at $x$ with a subspace in $V= T^\vee_x(W_t)$ of codimension $d$. Consider the relative Gauss map
$$X\setminus Z \;\longrightarrow\; \Gr(d, V )$$
which sends each point to the corresponding conormal space. This is a rational map whose locus of indeterminacy is precisely $Z$. Let $\gamma_X:\hat{X} \to \Gr (d,V)$ denote its resolution of indeterminacy which is obtained by blowing up the base locus~$Z\subset X$ as in~\cite[sect.~4.4]{Fulton}:
\[
\begin{tikzcd}
& \hat{X} \ar[dl, "\pi_X" above left] \ar[dr, "\gamma_X" above right] & \\
X \ar[rr, dashed] && \Gr(d, V)
\end{tikzcd}
\]
We want to control the image of the exceptional divisor $E_X=\pi_X^{-1}(Z) \subset \hat{X}$ under the map
\[
 \alpha_X \;=\; (\pi_X, \gamma_X): \; \hat{X} \,\to\, X\times \Gr(d, V) \,\subset \, W \times \Gr(d, V) 
 \;=\; \Gr (d,T^\vee(W/S)).
\]

\begin{lem} \label{lem:closed-embedding}
The morphism $\alpha_X$ is a closed embedding.
\end{lem}

\begin{proof}
By assumption $X\subset W$ is cut out by a regular sequence $f_1, \dots, f_d$. The same then holds for each fiber $X_t \subset W_t$. Hence it follows that the relative singular locus $\Sing(X/S)$ is cut out as a closed subscheme of $W$ by~$f_1, \dots, f_d$ and by the $d\times d$ minors 
\[
 J(i_1, \dots, i_d) \;:=\;
 \det
 \begin{pmatrix}
 \partial_{i_1} (f_1) & \cdots & \partial_{i_1}(f_d) \\
 \vdots & \ddots & \vdots \\
 \partial_{i_d} (f_1) & \cdots & \partial_{i_d}(f_d)
 \end{pmatrix} 
\] 
with $1\leq i_1 < \cdots < i_d \leq n$, where we fix an arbitrary basis $\partial_1, \dots, \partial_n \in V^\vee$ for the fiber of the relative tangent bundle and regard the basis vectors as relative derivations for the smooth morphism $W\to S$. 

\medskip

Now let $\iota:X\times \Gr(d, V)\hookrightarrow X\times \P(\Lambda^d V )$ be the Pl\"ucker embedding of the Grassmannian as a closed subvariety of projective space. We want to show that the composite
$$ \beta_X \;:=\; \iota \circ \alpha_X: \quad \hat{X} \;\longrightarrow\; X \times \Gr(d, V) \longrightarrow X \times \P (\textstyle\bigwedge^d V) 
$$ 
is a closed embedding. For this let $I\unlhd \mathcal{O}_X$ be the ideal sheaf of $Z\subset X$. Then we have \medskip
\[
 \hat{X} \;=\; \Proj_X R_I
 \quad \textnormal{for the graded Rees algebra} \quad 
 R_I \;:=\; \bigoplus_{n\geq 0} \; I^n \cdot t^n \;\subset\; \scrO_X[t],
\]
where $t$ is a dummy variable to keep track of degrees. The homomorphism\medskip
$$\beta_X^* \colon \quad \mathcal{O}_X\otimes \Sym^\bullet \left(\bigwedge^{d}V^{\vee} \right)\;\longrightarrow\;\; R_I \;=\; \bigoplus_{n\geq 0} \; I^n \cdot t^n $$
of graded $\mathcal{O}_X$-algebras satisfies \medskip 
\[
\beta_X^*(1\otimes (\partial_{i_1}\wedge \cdots \wedge \partial_{i_d})) \;=\; 
J(i_1, \dots, i_d)|_X \cdot t
\;\in\; I\cdot t \medskip 
\]
for $1\leq i_1 < \cdots < i_k\leq n$. But we have seen above that the $\mathcal{O}_X$-module $I$ is generated by the minors on the right hand side. Hence it follows that $\beta_X^*$ is an epimorphism in all degrees and so $\beta_X$ is a closed immersion.
\end{proof}

\begin{rem}
If $d=1$, then $\Gr(d, V)=\P V$ and the closed embedding $\alpha_X$ induces an isomorphism
\[
 \alpha_X: \quad \hat{X} \;\stackrel{\sim}{\longrightarrow} \; \P \Lambda_{X/S}. \medskip
\]
Indeed, the blowup $\hat{X}$ is again reduced and irreducible by~\cite[II.7.16]{Hartshorne}. Via $\alpha_X$ it is therefore an integral closed subscheme of $W\times \P V$ and as such it can be recovered as the Zariski closure of its restriction over the open dense subset $S\setminus \{s\}\subset S$, where it coincides with $\P \Lambda_{X/S}\subset W\times \P V$ by definition.
%
\end{rem}

In particular, one may look at the scheme-theoretic fiber of $\hat{X}$ over $s\in S(k)$ to compute the multiplicities in $\sp_s(\Lambda_{X/S})$. For $d>1$ the situation becomes more complicated, so in what follows we restrict ourselves to set-theoretic arguments. To pass back to conormal varieties we look at the projection
$
 \Fl(d, 1, V) \longrightarrow \Gr(d, V)
$
from the partial flag variety. This projection is a smooth equidimensional morphism of relative dimension $d-1$. On the fiber product $Y = \hat{X} \times_{\Gr(d, V)} \Fl(d,1,V)$ consider the morphism
\[
 \alpha_Y = (\pi_Y, \gamma_Y): \quad Y \;\longrightarrow\; X\times \P V
\]
where  $\pi_Y: Y\to \hat{X} \to X$ and $\gamma_Y: Y\to \Fl(d,1,V)\to \P V$ are the natural composite maps.
Taking the preimage of the previous exceptional divisor $E_X\subset \hat{X}$ we get the following lower bound on the specialization:

\begin{lem} \label{lem:upper-bound}
The preimage $E_Y = \pi_Y^{-1}(Z)\subset Y$ has dimension  $\dim(E_Y) = n$ and satisfies
\[ \alpha_Y(E_Y) \;\subset\; \P(\Supp(\sp_s(\Lambda_{X/S}))). \]
\end{lem}

\begin{proof}
The statement about the dimension holds because $\dim(Y)=n=\dim(V)$ and since the subvariety $E_Y\subset Y$ is a divisor, being the preimage of the exceptional divisor $E_X\subset \hat{X}$ under the fibration $Y\to \hat{X}$. To understand why the image $\gamma_Y(E_Y)$ is contained in the specialization, recall that $f: X\to S$ is smooth over the open subset $S^* = S\setminus \{s\}$. The identifications
\[
 Y^* \;:=\; S^* \times_S Y \;\simeq \; S^*\times_S \P \Lambda_{X/S} \;\subset\; S^* \times \P T^\vee(W/S)
\]
give the following Cartesian diagram where the vertical arrows are open embeddings and the top horizontal arrow is a closed immersion:
\[
\begin{tikzcd} 
Y^* \ar[r,hook] \ar[d, hook] & S^* \times_S \P T^\vee(W/S)\ar[d, hook] \\
Y \ar[r] & \P T^\vee(W/S)
\end{tikzcd}
\]
Now $\hat{X}$ is irreducible as a blowup of an irreducible variety. So $Y$ is irreducible as well, hence equal to the Zariski closure of its nonempty open subset $Y^* \subset Y$. But then  
\begin{align*}
 \bbP(\sp_s(\Lambda_{X/S})) &\;=\;
 \bigl( \textnormal{closure of} \;
 S^*\times_S \bbP \Lambda_{X/S} \;
 \textnormal{in} \;
 \P T^\vee(W/S) \bigr)_s \\
 &\;=\; \bigl( \textnormal{closure of the image of $Y^*$ in $\P T^\vee(W/S)$} \bigr)_s  
  \;\supseteq\; \alpha_Y(Y_s)
\end{align*} 
and we are done because by construction we have $E_Y \subseteq Y_s$. 
\end{proof}

\begin{cor} \label{cor:conormal-in-sp}
If $\alpha_Y: E_Y \to X\times \P V$ is generically finite onto its image, then we have
\[ \Lambda_Z \; \subset \; \Supp(\sp_s(\Lambda_{X/S})). \]
\end{cor}

\begin{proof}
Each irreducible component of $\P (\Supp(\sp_s(\Lambda_{X/S})))$ is the projectivization of some conormal variety. Each of them has dimension $n-1$, so lemma~\ref{lem:upper-bound} and our generic finiteness assumption imply that $\alpha_Y(E_Y)$ must appear as one of the components. But then this component is $\Lambda_Z$ because it maps onto $Z\subset X$.
\end{proof}

\medskip

Note that by lemma~\ref{lem:closed-embedding} the morphism $Y\to X\times \Fl(d,1,V)$ is a closed immersion and hence generically finite onto its image. So corollary~\ref{cor:conormal-in-sp} finishes the proof of proposition~\ref{prop:component-in-sp} since for codimension $d=1$ the morphism $\Fl(d,1,V)\to \P V$ is an isomorphism. This is the only point where we use $d=1$. For higher codimension the morphism $\alpha_Y: E_Y\to X\times \P V$ is not always a closed embedding, as the following example shows, but it may still be generically finite onto its image as needed for corollary~\ref{cor:conormal-in-sp}:

\begin{ex} \label{ex:higher-codim}
Let $W=\Spec \, k[x,y,z,s] \to S = \Spec \, k[s]$, and consider the family of subvarieties
\[
 X \;=\; \{ f = g = 0 \} \subset W
 \quad \textnormal{for} \quad 
 \begin{cases} 
 f \;=\; x^2 + y^2 + s, \\
 g \;=\; x^2 + z^2 - s.
 \end{cases} 
\]
Here $Z=\Sing(X/S)\subset X$ is a fat point with ideal sheaf $I=(xy,xz,yz)\unlhd \scrO_X$ and looking at the minors of the Jacobian matrix we see that the relative Gauss map is given in Pl\"ucker coordinates on the Grassmannian $\Gr(2, V)=\Proj\, k[w_1, w_2, w_3]$ by
\[
 \gamma_X: \quad X\setminus Z \;\longrightarrow\; \Gr(2, V) = \P^2,
 \quad (x,y,z,s) \;\mapsto\; [w_1:w_2:w_3] = [yz:xz:-xy].
\]
Note that the right hand side does not involve the parameter $s$. Furthermore, we have $(2x^2 + y^2 + z^2)|_X = (f + g)|_X = 0$ and hence the relative Gauss map $\gamma_X$ factors over
\[
   Q_X \;=\; \{ 2w_2^2w_3^2  + w_1^2w_3^2 + w_2^2w_3^2 \;=\; 0 \} \;\subset\; \Gr(2, V). \medskip
\] 
Write $\P V = \Proj \; k[v_1, v_2, v_3]$ for the dual coordinates $v_i$ where the flag variety is given by
\[
 \Fl(2,1,V) \;=\; \{ v_1w_1 + v_2w_2 + v_3w_3 = 0 \} \;\subset\; \Gr(2, V)\times \P V
\]
then for 
\[
 Q_Y \;=\; \{ 2w_2^2w_3^2  + w_1^2w_3^2 + w_2^2w_3^2 = v_1w_1 + v_2w_2 + v_3w_3 = 0 \} \;\subset\; \Gr(2,V) \times \P V.
\]
we get the following diagram where the squares are Cartesian and the hooked arrows are closed immersions:
\[
\begin{tikzcd}
 E_Y \ar[r, hook] \ar[d] & Y \ar[r, hook] \ar[d] & X\times Q_Y \ar[r, hook] \ar[d] & X \times \Fl(2,1,V) \ar[r] \ar[d] & X\times \P V \\
  E_X \ar[r, hook] & \hat{X} \ar[r, hook] & X\times Q_X \ar[r, hook] & X \times \Gr(2,V) &
\end{tikzcd}
\]
The composite of the arrows in the top row is the morphism $\alpha_Y: E_Y \to X\times \P V$ that we are interested in. The diagram shows that it is not a closed immersion, over $Z^\mathrm{red} = \{0\} \subset X$ we have the factorization
\[ 
\begin{tikzcd}
E_Y^\mathrm{red} \arrow[rr, "\alpha_Y"] \arrow[dr] & &  \P V \\
&  Q_Y\arrow[ur, swap, "4:1"] &
\end{tikzcd}
\]
where $Q_Y \to \P V$ is an irreducible cover of generic degree four! However, since the left diagonal arrow is a closed immersion and hence birational for dimension reasons, the morphism $\alpha_Y: E_Y \to X\times \P V$ is generically finite over its image.
\end{ex}

\medskip 

\section{Generalities about families of rational maps}\label{sec:gen}

Before we apply the above to Gauss maps, let us recall some generalities about families of rational maps. Let $f\colon X\to S$ be a faithfully flat morphism of varieties of relative dimension $n$ with irreducible fibers. Let $\scrL\in \Pic(X)$ be a line bundle and $\scrV$ a rank $n+1$ vector subbundle of $ f_*\scrL$. Then for each point $s\in S(k)$ we get a linear series~$\scrV_s \subset H^0(X_s, \scrL_s)$ and we denote by $\phi_s\colon X_s \dashrightarrow \mathbb{P}\scrV_s$ the corresponding rational map. Note that since the source and the target of this map have the same dimension, the map is a generically finite cover iff it is dominant. So we consider the degree map 
\[ \deg: \quad S(k) \;\longrightarrow\; \bbN_0, \quad  s\;\mapsto \;\deg(\phi_s), 
\]
where we put $\deg(\phi_s)=0$ if $\phi_s$ is not dominant.

\begin{lem}\label{lem:const}
The degree map $\deg$ is constructible.
\end{lem}

\begin{proof}
By~\cite[prop.~4.4]{Fulton} we can compute the degree in terms of Segre and Chern classes as
$$
\deg(\phi_s) \;=\; \int_{X_s}c_1(\scrL_s)^n-\int_{B_s}c(\scrL_s)^n\cap s(B_s,X_s)
$$
where $B_s\subset X_s$ denotes the base locus of the linear series $\mathcal{V}_s\subset H^0(X_s, \scrL_s)$. We can put together all these fiberwise base loci into a relative base locus and consider the flattening stratification of this relative base locus. This is a stratification of~$S$ such that on each stratum the above intersection number is constant, hence the function $\deg$ is constructible. 
\end{proof}

However, in general the degree is neither upper nor lower semi-continuous, as the following variation of \cite[ex.~2.3]{CGS} shows:

\begin{ex} \label{ex:not-semicontinuous}
Let $S$ be a smooth affine curve with two marked points $s_\pm \in S(k)$ 
and fix positive integers $n_\pm \leq n < 27$. 
Let $p_j: S\longrightarrow \P^3$ for $j=1,\dots, n$ be such that 
\begin{itemize} 
\item for $t\neq s_\pm$ the points $p_j(t)$ are in general position, \smallskip 
\item for $t=s_\pm$ they consist of $n_\pm$ general points on a given line $\ell$ and~$n-n_\pm$ points in general position not on that line. \smallskip 
\end{itemize} 
For $t\in S(k)$ let $f_t\colon \mathbb{P}^3\dashrightarrow \mathbb{P}^3$ be the generically finite rational map defined by a linear system of four generic cubics passing trough the $p_j(t)$. By loc.~cit.~its degree is
\[
 \deg(f_t) \;=\; 
 \begin{cases} 
 27 - n & \textnormal{for $t\neq s_\pm$}, \\
 20 - (n-n_\pm) & \textnormal{for $t=t_\pm$ and $n_\pm \geq 4$},
 \end{cases} 
\] 
since in the second case the indeterminacy locus of~$f_\pm$ consists of the chosen line $\ell$ together with the remaining $n-n_\pm$ points. So the degree is a constructible function as predicted by lemma \ref{lem:const}. However, taking for example $(n,n_+,n_-)=(20,10,4)$ we obtain an example where the generic value of the degree is seven, jumps up to ten at one point and down to four at another point. Hence in the same family the degree can both decrease and increase under specialization.
\end{ex}

\medskip

\section{Gauss maps on abelian varieties} \label{sec:abvar}

We now apply the above to an abelian scheme $f: A\to S$, so in this section we take $W=A$. Let $X\subset A$ be a closed subvariety which is flat over $S$. For $s\in S(k)$ we have the Gauss map
\[
 \gamma_{X_s}: \quad  \bbP \Lambda_{X_s} \;\longrightarrow \;\bbP V 
\]
where $V=H^0(A_s, \Omega_{A_s}^1)$. If this map is dominant, then for dimension reasons it is a generically finite cover and we denote by $\deg(\gamma_{X_s})$ its generic degree. If the map is not dominant we put $\deg(\gamma_{X_s})=0$; this happens iff the subvariety $X_s\subset A_s$ arises by pull-back from some smaller dimensional abelian quotient variety~\cite[th.~1]{WeissauerDegenerate}, which by~\cite[th.~3]{Abramovich} happens iff $X_s$ is not of general type. The degree of the above Gauss map is related to the degree of conormal varieties as follows:

\begin{lem}\label{lem:degreGauss}
We have $\deg(\Lambda_{X_s})=\deg(\gamma_{X_s})$.
\end{lem}

\begin{proof}
The degree of our Gauss map $\gamma_{X_s}: \bbP \Lambda_{X_s} \to \bbP V $ coincides with the degree considered by Franecki and Kapranov terms of tangent rather than cotangent spaces in~\cite[sect.~2]{FKGauss}: Up to the duality $\Gr(d, V)\simeq \Gr(g-d, V^\vee)$ they study the map $p\circ q$ defined by the  diagram
\[
\begin{tikzcd}
 \tilde{X}_s \ar[d] \ar[r, "q"] & F(d,1, V) \ar[d] \ar[r, "p"] & G(1, V)=\bbP V \\
 \Sm(X_s) \ar[r] & G(d, V) &
\end{tikzcd}
\]
where $\tilde{X}_s = \Sm(X_s) \times_{G(d,V)} F(d,g-1, V)$. By construction $\tilde{X}_s \subset \bbP \Lambda_{X_s}$ is an open subset of the conormal variety, and their map $p\circ q$ is the restriction of our Gauss map to this open subset. Hence the claim follows from~\cite[prop.~2.2]{FKGauss}. 
\end{proof}

In particular $\deg(\Lambda_{X_s}) \geq 0$. Together with the preservation of the total degree under Lagrangian specialization this leads to our first semicontinuity result:

\begin{cor}\label{cor:semicon}
The map $S(k) \to \bbN_0, s\mapsto \deg(\gamma_{X_s})$ is lower semicontinuous.
\end{cor}
\begin{proof}
By Lemma \ref{lem:const}, we know that the map is constructible. We have to show that its values decreases under specialization. For this we may assume that $S$ is a curve, and after base change to its normalization we may assume this curve to be smooth. Let $d$ be the value from proposition~\ref{prop:degree-constant}. With notations as in lemma~\ref{lem:specialization} then
\[
 \deg(\Lambda_{X_s}) \;=\;
 \begin{cases}
 d & \textnormal{if $s\notin \Sigma$}, \\
 d - \delta(s) & \textnormal{if $s\in \Sigma$},
 \end{cases} 
\]
where $\delta(s) = \sum_{Z\subset X_s} m_Z\cdot \deg(\Lambda_Z)$ with multiplicities $m_Z \geq 0$. Now in the case of abelian varieties the occuring degrees coincide with the degrees of the corresponding Gauss maps by lemma~\ref{lem:degreGauss}. In particular, since the degrees of Gauss maps are obviously nonnegative, we have $\deg(\gamma_Z)\geq 0$ and therefore $\delta(s) \geq 0$, which proves that the degree of the Gauss map is constant on an open dense subset and can only drop on the finitely many points of the complement. \end{proof}

To see where the function in the previous corollary actually jumps, recall that a subvariety $Z\subset A_s$ has Gauss degree $\deg(\gamma_Z)=0$ iff it is not of general type. Thus we obtain the following sufficient jumping criterion:

\begin{cor}\label{cor:strict_semicon}
Suppose $\dim(S)=1$. Let $X\subset A$ be a divisor which is flat over~$S$ and let~$0\in S(k)$ be a point such that $\Sing(X_0)$ has an irreducible component $Z$ which is of general type and not contained in the closure of $\bigcup_{t\neq 0}\Sing(X_t)$. Then 
there is an open dense subset $U\subset S$ such that\medskip
$$\deg(\gamma_{X_t}) -\deg(\gamma_{X_0}) \;\geq\; \deg(\gamma_Z) \;>\; 0
\quad \textnormal{\em for all} \quad t\;\in\; U(k) \setminus \{0\}.$$
\end{cor}

\begin{proof}
The first inequality follows from prop.~\ref{prop:component-in-sp}, the second from our assumption that the subvariety $Z\subset A_0$ is of general type. 
\end{proof}

\medskip 

\section{Application to the Schottky problem} \label{sec:schottky} 

The moduli space $\scrA_g$ of principally polarized abelian varieties of dimension~$g$ over the field $k$ admits the finite filtration 
$\cdots \subseteq \scrG_d \subseteq \scrG_{d-1} \subset \cdots \subseteq \scrG_{g!} = \scrA_g$
by the {\em Gauss loci}
$$
\scrG_d \;:= \; \{(A,\Theta)\in \scrA_g \, | \, \deg(\gamma_{\Theta}) \leq d\}
\;\subseteq\; \scrA_g.
$$
Our semicontinuity result implies:

\begin{cor}\label{cor:closed}
For any $d\in \bbN$ the Gauss loci $\scrG_d$ are closed in $\scrA_g$.
\end{cor}

\begin{proof}
The moduli space $\scrA_g$ has a finite cover by a smooth quasi-projective variety over which there exists universal theta divisor. On this cover the Gauss maps fit together in a family of rational maps as in the setting of section \ref{sec:gen}. The Gauss loci are the level sets of the degree map, and we have to show that this map is lower-semicontinuous. This follows from corollary \ref{cor:semicon}. 
\end{proof}

Using our sufficient criterion for jumps in the degree of Gauss maps, we can now show that the stratification by the Gauss loci refines the stratification by the Andreotti-Mayer loci
\[
 \scrN_c \;=\; \{ (A, \Theta) \in \scrA_g \mid \dim \Sing(\Theta) \geq c\}
\]
from~\cite{AM}. Some care is needed because the singular locus of the theta divisor may have components which are negligible, i.e.~not of general type:

\begin{rem}
There are indecomposable ppav's $(A,\Theta)\in \scrA_g$ with a theta divisor for which $\Sing(\Theta)\subset A$ is negligible. This even happens for generic ppav's on certain irreducible components of Andreotti-Mayer loci: For instance, for~$g=5$ one can show that for a generic ppav on the component $\scrE_{5,1}\subset \scrN_1$ from~\cite[thm.~4.1(ii)]{DebarreCodimension3} the singular locus $\Sing(\Theta)$ is an elliptic curve. 
\end{rem}

A more detailed discussion will be given in a forthcoming work by Constantin Podelski. In any case negligible components can only appear on decomposable abelian varieties, hence the following corollary of our jumping criterion covers all Andreotti-Mayer strata whose general point is a simple abelian variety:

\begin{cor} \label{cor:AM}
Let $c\in \bbN$, and let $\scrN \subset \scrN_c$ be an irreducible component whose general point is a ppav whose singular locus of the theta divisor has no negligible components. Then~$\scrN$ is an irreducible component of $\scrG_d$ for some $d\in \bbN$.
\end{cor}

\begin{proof}
Let $s \in \scrN(k)$ be a general point on the given component. Since the moduli space of ppav's is a quasiprojective variety, we may pick an affine curve $S\subset \scrA_g$ such that $S\cap \scrN_c = \{s\}$ and $S$ meets $\scrN$ transversely. After passing to a finite cover we may assume that there exists an abelian scheme $f: A\to S$ and a universal theta divisor $\Theta \subset A$ over this curve, and by our choice of the curve we have \smallskip
\[
 \dim \Sing(\Theta_t) \; < \; \dim \Sing(\Theta_{s})
 \quad \textnormal{for all $t\neq s$}. 
\]
Hence $\Sing(\Theta_s)$ has an irreducible component not in the closure of $\bigcup_{t\neq s} \Sing(\Theta_t)$ and so
\[
 \deg(\gamma_{\Theta_t}) \;>\; \deg(\gamma_{\Theta_s})
  \quad \textnormal{for all $t\neq s$} 
\]
by corollary~\ref{cor:strict_semicon}. Varying $s$ in an open subset of the component $\scrN$ and varying $S$ among all curves meeting this component transversely in the chosen point, we get that some nonempty open subset of $\scrN$ is also an open subset of a Gauss locus $\scrG_d$ for some $d\in \bbN$. Hence the result follows by passing the the closure, since both the Andreotti-Mayer loci and the Gauss loci are closed in $\scrA_g$.
\end{proof}

As an application we get that the stratification by the degree of the Gauss map gives a solution to the Schottky problem as conjectured in~\cite{CGS}, where for the Prym version we denote by $D(g)$ the degree of the varieties of quadrics in $\bbP^{g-1}$ of rank at most three:

\begin{cor} \label{cor:Schottky}
We have the following components of Gauss loci in $\scrA_g$: \medskip 

\noindent (a) The locus of Jacobians is a component of $\scrG_{d}$ for $d=\binom{2g-2}{g-1}$.\medskip

\noindent (b) The locus of hyperelliptic Jacobians is a is a component of $\scrG_{d}$ for $d=2^{g-1}$.\medskip

\noindent (c) The locus of Prym varieties is a is a component of $\scrG_d$ for $d=D(g)+2^{g-3}$.\medskip
\end{cor}

\begin{proof}
The locus of Jacobians is a component of an Andreotti-Mayer locus by~\cite{AM}, and by \cite[Proposition 3.4]{Pirola} a general Jacobian variety is a simple abelian variety and thus in particular has no negligible subvarieties other than itself. Furthermore, it is well-known that the degree of the Gauss map of a Jacobian is $d=\binom{2g-2}{g-1}$, see e.g.~\cite[proof of prop.~10]{Andreotti}. Hence part (a) follows from corollary \ref{cor:AM}. If we replace Jacobians by hyperelliptic Jacobians, the above arguent works also in the hyperelliptic case, with the same references.
For Prym varieties the argument is again the same but now one has to replace reference~\cite{AM} with \cite{DebarrePrym}, and reference~\cite{Andreotti} with \cite[Main Theorem]{VerraGauss}. In the last reference the reader can also find an explicit expression for the number $D(g)$.
\end{proof}

\medskip 

\section{A topological view on jump loci} \label{sec:IC}

In this section we work over the complex numbers with the Euclidean topology. 
Let $W$ be a smooth complex projective variety. For a closed subvariety~$X\subset W$ the singular locus $\Sing(X)$ and the conormal degree $\deg(\Lambda_X)$ are not topological invariants of the subvariety, as example~\ref{ex:local-euler} shows. But both are related to the intersection cohomology $\mathrm{IH}^\bullet(X)$ which only depends on the homeomorphism type of the subvariety in the Euclidean topology; see~\cite{BorelIC, GMIntersectionI,GMIntersectionII, KirwanWoolf, MaximIHP}. The Euler characteristic\bigskip
\[
 \chi_\mathrm{IC}(X) \;=\; \sum_{i\geq 0} \; (-1)^{i+\dim(X)} \dim \mathrm{IH}^i(X)
\]
%
can be read off from a generalization of the Gauss-Bonnet theorem: The Kashiwara index formula~\cite[th.~9.1]{GinzburgCharacteristic} writes it as a degree in the sense of definition~\ref{defn:degree}. More precisely
\[ \chi_\IC(X) \;=\; \deg(\CC(\delta_X)), \]
where $\delta_X\in \Perv(W)$ denotes the perverse intersection complex of $X\subseteq W$ \cite{BBD, DCM} and where the characteristic cycle $\CC(\delta_X) \in \scrL(W)$ is an effective conic Lagrangian cycle  which contains $\Lambda_X$ as a component of multiplicity one but may also have as components the conormal varieties to certain $Z\subseteq \Sing(X)$. Passing from conormal varieties to characteristic cycles restores topological invariance of the degree:

\begin{ex} \label{ex:cusp}
In $W=\bbP^2$ the conormal degree for a smooth rational curve differs from the one for a cuspical cubic, see example~\ref{ex:local-euler}. But this is compensated by a difference in
\[
 \CC(\delta_X) \;=\; 
 \begin{cases}
 \Lambda_X & \textnormal{if $X$ is a smooth rational curve}, \\
 \Lambda_X + \Lambda_{\{p\}} & \textnormal{if $X$ is a cuspidal cubic with cusp $p$}.
 \end{cases}
\]
which in both cases gives the total degree $\deg(\CC(\delta_X))=-2$.
\end{ex}

In what follows we want to understand how for a morphism $X\to S$ of complex varieties the intersection cohomology Euler characteristic of the fibers varies. Basic stratification theory implies:

\begin{lem}
The map $s\mapsto \chi_\IC(X_s)$ is constructible.
\end{lem}

\begin{proof}
The homeomorphism invariance of intersection cohomology~\cite[th.~4.1]{GMIntersectionII} shows that for a topologically locally trivial fibration over a connected base, all fibers have the same intersection cohomology. Any morphism of complex algebraic varieties restricts to a topologically locally trivial fibration over a Zariski open dense subset of the target~\cite[cor.~5.1]{VerdierStratifications}, so we are done by Noetherian induction. 	
\end{proof}

In general the map in this lemma is neither upper nor lower semicontinuous, the jumps may go in both directions:

\begin{ex} \label{ex:IC-not-semicontinuous}
Let $Q\subset \bbP^3$ be a quadric. Then by~\cite[ex.~2.3.21 and th.~2.4.6]{MaximIHP} we have
\[
 \chi_\IC(Q) \;=\; 
 \begin{cases}
 4 & \textnormal{if $Q\simeq \bbP^1 \times \bbP^1$ is smooth}, \\
 3 & \textnormal{if $Q$ is a cone over a smooth rational curve}, \\
 6 & \textnormal{if $Q$ is a union of two projective planes}.
 \end{cases}
\]
So for a family of quadrics whose general member is smooth, the number $\chi_\IC(Q)$ jumps down on nodal quadrics but jumps up on reducible quadrics.
\end{ex} 

Note that here the size of the jumps is precisely the Euler characteristic of the singular locus. This fits with the following sheaf-theoretic version of the Lagrangian specialization principle:

\begin{thm} \label{thm:nearby-cycles}
Let $f: W\to S$ be a smooth proper family over a curve $S$, and let~$X\subset W$ be a closed subvariety such that the morphism $f: X\to S$ is flat with generically reduced fibers. Then there exists $d\in \bbZ$ and a finite subset~$\Sigma \subset S$ such that\medskip
\[
 \chi_\IC(X_s) \;=\;
 \begin{cases}
 d & \textnormal{for $s\in S\setminus \Sigma$}, \\
 d - \deg(\Lambda(s)) & \textnormal{for $s\in \Sigma$},
 \end{cases}  \medskip
\]
where $\Lambda(s) = \sp_s(\CC(\delta_X)) - \CC(\delta_{X_s}) \in \scrL(W_s)$ is an effective cycle.
\end{thm}

\begin{proof} 
We interpret Lagrangian specialization via perverse sheaves. For $s\in S(\bbC)$ one has the functor of nearby cycles $\Psi_s: \Perv(W)\to \Perv(W_s)$, which is an exact functor with
\[
 \CC(\Psi_s(P)) \;=\; \sp_s(\CC(P))
 \quad \textnormal{for all} \quad P \;\in\; \Perv(W)
\]
by~\cite[th.~5.5]{GinzburgCharacteristic}. Here we abuse notation and view $\CC(P)$ as an element of the group of relative conic Lagrangian cycles $\scrL(W/S)$ via remark~\ref{rem:bmm}, discarding any component that is not flat over $S$. The last part of the specialization lemma~\ref{lem:specialization} has a sheaf-theoretic version: For any closed $S$-flat subvariety $X\subset W$ such that the map $f: X\to S$ has generically reduced fibers, there exists a finite subset $\Sigma \subset S$ such that the semisimplification $(\Psi_s(\delta_X))^\mathrm{ss}$ of the perverse sheaf $\Psi_s(\delta_X)$ has the form
\[ 
(\Psi_s(\delta_X))^\mathrm{ss}
\;\simeq\;
\begin{cases} 
\delta_{X_s} & \textnormal{for $s\in S\setminus \Sigma$}, \\
\delta_{X_s} \oplus P(s) & \textnormal{for $s\in \Sigma$},
\end{cases}   \medskip 
\]
where $P(s)\in \Perv(X_s)$ is a perverse sheaf with support contained in $\Sing(X_s)$. So we get
\[
\sp_s(\CC(\delta_X)) \;=\; 
\CC(\Psi_s(\delta_X))
\;=\;
\begin{cases} 
\CC(\delta_{X_s}) & \textnormal{for $s\in S\setminus \Sigma$}, \\
\CC(\delta_{X_s}) + \Lambda(s) & \textnormal{for $s\in \Sigma$},
\end{cases}   \medskip 
\]
where $\Lambda(s) = \CC(P(s))\in \scrL(W_s)$ is effective, being the characteristic cycle of a perverse sheaf. Hence the result follows by noting that if $f: W\to S$ is proper, then by proposition~\ref{prop:degree-constant} the degree $d=\deg(\sp_s(\CC(\delta_X)))$ is independent of $s$. 
\end{proof}

In the case of abelian varieties the positivity of conormal degrees then gives an analog of theorem~\ref{thm:semicontinuous-in-any-codim}. The same argument works for a much wider class of varieties, we only need the following positivity property:

\begin{defn} 
A variety $X$ satisfies the {\em signed Euler characteristic property} if we have \medskip
\[
 \chi(X, P) \;:=\; \sum_{i\in \bbZ} (-1)^i \dim \mathrm{H}^i(X, P) \;\geq\; 0
 \quad \textnormal{for all} \quad P\;\in\; \Perv(X).
\]
\end{defn}

\noindent The terminology is borrowed from~\cite{EGM}. The above property holds for semiabelian varieties~\cite{FKGauss} and hence also for any finite cover of closed subvarieties of them:

\begin{lem} \label{lem:signed-Euler}
If a variety $A$ has the signed Euler characteristic property, then so does any variety with a finite morphism to $A$. In particular, any variety with a finite morphism to a semiabelian variety has the signed Euler characteristic property.
\end{lem} 

\begin{proof} 
If $f: X\to A$ is a finite morphism, then for any perverse sheaf $P\in \Perv(X)$ the direct image is a perverse sheaf $Rf_*(P)\in \Perv(A)$. If $A$ has the signed Euler characteristic property, which holds for instance for abelian varieties~\cite{FKGauss}, then we get $\chi(A, P)=\chi(A, Rf_*(P))\geq 0$. 
\end{proof}

The above theorem shows that for any family of such varieties the intersection cohomology Euler characteristic is semicontinuous:

\begin{cor} \label{cor:IC-semicontinuous}
Let $f: W\to S$ be a smooth proper morphism to a variety $S$, and let $X\subset W$ be a closed subvariety such that $f: X\to S$ is flat and all its fibers are generically reduced and have the signed Euler characteristic property. 
Then for each $d\in \bbN$ the subsets
$
S_d = \bigl\{ s\in S \mid \chi_{\mathrm{IC}}(X_s) \leq d \bigr\} \subseteq  S
$ 
are Zariski closed.
\end{cor}

\begin{proof} 
We must show that $\chi_\IC(X_s)$ cannot increase under specialization. For this we can assume~$S$ is a smooth curve. Then theorem~\ref{thm:nearby-cycles} applies, here $\chi(X_s, P(s))\geq 0$ since $X_s$ has the signed Euler characteristic property.
\end{proof}

In deciding where the Euler characteristic actually jumps, we need to be more careful. Proposition~\ref{prop:component-in-sp} gives a way to see extra components in $\sp_s(\CC(\delta_X))$ but does not guarantee that these enter in a new summand $\Lambda(s)$, a priori they could also appear in $\CC(\delta_{X_s})$; however, this second case can only happen if $\CC(\delta_{X_s})$ is reducible, which one can often exclude by a direct computation.

\medskip 

Let us illustrate this again with theta divisors. Corollary~\ref{cor:IC-semicontinuous} says that for $d\in \bbN$ the loci
\[
\scrX_d \;=\; \{ (A, \Theta) \in \scrA_g \mid \chi_\IC(\Theta) \leq d\} \;\subseteq\; \scrA_g. \medskip
\] 
are closed, and by the homeomorphism invariance of intersection cohomology they only depend on the topology of the theta divisor. This provides a topological view on Andreotti-Mayer loci, for instance:

\begin{cor} \label{cor:IC-AM}
Let $\scrN \subset \scrN_c$ be an irreducible component of an Andreotti-Mayer locus such that a general point of this component is a ppav $(A, \Theta)$ with the property that
\begin{itemize} 
\item $\CC(\delta_\Theta)$ is irreducible, and\smallskip 
\item $\Sing(\Theta)$ has no negligible components. \smallskip
\end{itemize} 
Then $\scrN$ is also an irreducible component of $\scrX_d$ for some $d\in \bbN$.
\end{cor}

\begin{proof} 
Use the same argument as in corollary~\ref{cor:AM}, together with the remark after the proof of corollary~\ref{cor:IC-semicontinuous}.
\end{proof}

This in particular applies to the locus of Jacobians. In the following corollary we do not mention hyperelliptic Jacobians because for them $\CC(\delta_\Theta)$ is reducible, and we haven't checked what happens for a generic Prym variety. However, we include the Andreotti-Mayer locus $\scrN_0 \subset \scrA_g$ of ppav's with a singular theta divisor:

\begin{cor} \label{cor:schottky-topology}
Inside the moduli space $\scrA_g$ we have:\smallskip
\begin{enumerate} 
\item The locus $\scrN_0$ is equal to $\scrX_d$ for $d=\begin{cases} g!-1 & \textnormal{\em if $g$ is odd},\\ g! - 2 & \textnormal{\em if $g$ is even}. \end{cases}$  \medskip 
\item The locus of Jacobians is a component of $\scrX_{d}$ for $d=\binom{2g-2}{g-1}$.
\end{enumerate} 
\end{cor}

\begin{proof} 
(1) By definition $\scrA_g\setminus \scrN_0$ consists of all ppav's $(A,\Theta)$ with a smooth theta divisor and for those we know that $\chi_\IC(\Theta)=g!$ because for a smooth variety intersection cohomology equals Betti cohomology. But at a generic point $(A, \Theta)$ of each of the two components of $\scrN_0$ the theta divisor has one respectively two nodes, and then
\[
 \chi_\IC(\delta_\Theta) \;=\;
 \begin{cases}
 g! - 2k & \textnormal{if $g$ is even}, \\
 g! - k & \textnormal{if $g$ is odd},
 \end{cases}
\] 
where $k\in \{1,2\}$ denotes the number of nodes~\cite[proof of prop.~4.2(2)]{KrWSchottky}. Hence the claim follows by corollary~\ref{cor:IC-semicontinuous}. Note that the degree of the classical Gauss map is $\deg(\Lambda_\Theta)=g! - 2k$ in both cases, but for odd $g$ the cycle $\CC(\delta_\Theta)=\Lambda_\Theta + \Lambda_{\Sing(\Theta)}$ is reducible and we cannot directly apply corollary~\ref{cor:IC-AM}.

\bigskip

(2) For Jacobians of nonhyperelliptic curves we know that $\CC(\delta_\Theta)=\Lambda_\Theta$ is irreducible by~\cite[th.~3.3.1]{BrB}, so if we specialize to such a Jacobian, then any new component of the specialization must enter in $\Lambda(s)$. So the same argument as in the proof of corollary~\ref{cor:Schottky} shows that the locus of Jacobians is a component of~$\scrX_d$ where~$d=\chi_\IC(\Theta)=\deg(\CC(\delta_\Theta))=\deg(\Lambda_\Theta)$ is the degree of the Gauss map for the theta divisor on a general Jacobian variety as in corollary~\ref{cor:Schottky}.\end{proof}

The above is still only a weak solution to the Schottky problem, though~$\chi_\IC(\Theta)$ also appears as the dimension of an irreducible representation of a certain reductive group which gives more information~\cite[sect.~4]{KraemerMicrolocalII}. The following example for $g=4$ illustrates the difference between the various numerical invariants:

\begin{ex} 
Let $(A, \Theta) \in \scrA_4$.\smallskip
\begin{itemize}[leftmargin=2.2em,labelindent=16pt] 
\item If $\Sing(\Theta)$ consists of $8$ nodes, then $\deg(\gamma_\Theta)=\chi_\IC(\delta_\Theta) = 8$.\smallskip 
\item If $\Sing(\Theta)$ consists of $5$ nodes, then $\deg(\gamma_\Theta)=\chi_\IC(\delta_\Theta) = 14$.\smallskip
\item If $(A, \Theta)$ is a hyperelliptic Jacobian, then $\deg(\gamma_\Theta)=8$ and $\chi_\IC(\delta_\Theta) = 14$.\smallskip
\end{itemize} 
So there are non-homeomorphic theta divisors whose Gauss maps have the same degree. Are there also homeomorphic theta divisors with different Gauss degrees?
\end{ex}

\bigskip 

%
%
%
%


\bibliographystyle{amsplain}
\bibliography{Bibliography}

\providecommand{\bysame}{\leavevmode\hbox to3em{\hrulefill}\thinspace}
\providecommand{\MR}{\relax\ifhmode\unskip\space\fi MR }
\providecommand{\MRhref}[2]{%
  \href{http://www.ams.org/mathscinet-getitem?mr=#1}{#2}
}
\providecommand{\href}[2]{#2}
\begin{thebibliography}{10}

\bibitem{Abramovich}
{Abramovich, D.}, \emph{{Subvarieties of semiabelian varieties}},
  {Compos.~Math.} \textbf{{90}} ({1994}), {37--52}.

\bibitem{Andreotti}
{Andreotti, A.}, \emph{On a theorem of {T}orelli}, Amer. J. Math. \textbf{80}
  (1958), 801--828.

\bibitem{AM}
{Andreotti, A. and Mayer, A. L.}, \emph{{On period relations for abelian
  integrals on algebraic curves}}, Ann. Sc. Norm. Super. Pisa Cl. Sci.
  \textbf{21} (1967), 189--238.

\bibitem{AC}
{Auffarth, R. and Codogni, G.}, \emph{{Theta divisors whose Gauss map has a
  fiber of positive dimension}}, Journal of Algebra \textbf{548} (2020), 153 --
  161.

\bibitem{BBD}
{Beilinson, A., Bernstein, J. and Deligne, P.}, \emph{{Faisceaux Pervers}},
  Ast{\'e}risque \textbf{100} (1982).

\bibitem{BorelIC}
{{Borel, A. et al.}}, \emph{{{Intersection Cohomology}}}, {{Progress in
  Math.}}, vol.~{{50}}, {{Birkh\"auser}}, {{1984}}.

\bibitem{BrB}
{Bressler, P. and Brylinski, J.-L.}, \emph{{On the singularities of theta
  divisors on jacobians}}, J. Algebraic Geom. \textbf{7} (1998), 781--796.

\bibitem{BMM}
{Brian\c{c}on, J., Maisonobe, P. and Merle, M.}, \emph{{Localization de
  syst\`emes diff\'erentiels, stratifications de Whitney et conditions de
  Thom}}, {Invent. Math.} \textbf{{117}} ({1994}), {531--550}.

\bibitem{BrieskornDiffTop}
{Brieskorn, E.}, \emph{{Beispiele zur Differentialtopologie von
  Singularit\"aten}}, {Invent. Math.} \textbf{{2}} ({1966}), {1--14}.

\bibitem{BrieskornSingularNormal}
\bysame, \emph{{Examples of singular normal complex spaces which are
  topological manifolds}}, {Proc.~Nat.~Acad.~Sci.~U.S.A.} \textbf{{55}}
  ({1966}), {1395--1397}.

\bibitem{CGS}
{Codogni, G., Grushevsky, S. and Sernesi, E.}, \emph{{The degree of the Gauss
  map of the theta divisor}}, {Algebra \& Number Theory} \textbf{{11}}
  ({2017}), {983--1001}.

\bibitem{DCM}
{De Cataldo, M. and Migliorini, L.}, \emph{{The decomposition theorem, perverse
  sheaves and the topology of algebraic maps}}, {Bull. Amer. Math. Soc.}
  \textbf{{46}} ({2009}), {535--633}.

\bibitem{DebarreCodimension3}
{Debarre, O.}, \emph{{Sur les vari\'et\'es ab\'eliennes dont le diviseur theta
  est singulier en codimension 3}}, {Duke Math.~J.} \textbf{{57}} ({1988}),
  {221--273}.

\bibitem{DebarrePrym}
\bysame, \emph{Vari\'et\'es de {P}rym et ensembles d'{A}ndreotti et {M}ayer},
  Duke Math. J. \textbf{60} (1990), no.~3, 599--630.

\bibitem{EGM}
{Elduque, E., Geske, C. and Maxim, L.}, \emph{{On the signed Euler
  characteristic property for subvarieties of abelian varieties}}, {J. of
  Sing.} \textbf{{17}} ({2018}), 368--387.

\bibitem{FKGauss}
{Franecki, J. and Kapranov, M.}, \emph{{The Gauss map and a noncompact
  Riemann-Roch \mbox{formula} for constructible sheaves on semiabelian
  varieties}}, {Duke Math. J.} \textbf{104} (2000), 171--180.

\bibitem{Fulton}
{Fulton, W.}, \emph{{Intersection theory (second ed.)}}, Springer Verlag, 1998.

\bibitem{FKM}
{Fulton, W., Kleiman, S. and MacPherson, R.}, \emph{{About the enumeration of
  contacts}}, {Algebraic Geometry -- Open Problems. Proceedings Ravello 1982}
  ({Ciliberto, C., Ghione, F., and Orecchia, F.}, ed.), {Lecture Notes in
  Math.}, vol. {997}, {Springer}, {1983}, pp.~{156--196}.

\bibitem{GinzburgCharacteristic}
{Ginzburg, V.}, \emph{{Characteristic varieties and vanishing cycles}}, Invent.
  Math. \textbf{84} (1986), 327--402.

\bibitem{GMIntersectionI}
{{Goresky, M. and MacPherson, R.}}, \emph{{{Intersection homology theory}}},
  {{Topology}} \textbf{{{19}}} ({{1980}}), {{135--162}}.

\bibitem{GMIntersectionII}
\bysame, \emph{{Intersection homology II}}, {Invent. Math.} \textbf{{71}}
  ({1983}), {77--129}.

\bibitem{Hartshorne}
{Hartshorne, R.}, \emph{{Algebraic geometry}}, {Grad. Texts in Math.}, vol.~52,
  Springer Verlag, 1977.

\bibitem{Kennedy}
{{Kennedy, G.}}, \emph{{{MacPherson's chern classes of singular algebraic
  varieties}}}, {{Communications in Alg.}} \textbf{{18}} ({1990}),
  {2821--2839}.

\bibitem{KirwanWoolf}
{{Kirwan, F. and Woolf, J.}}, \emph{{{An introduction to Intersection Homology
  Theory (2nd ed)}}}, {{Chapman and Hall}}, {{2006}}.

\bibitem{KleimanTangencyAndDuality}
{{Kleiman, S. L.}}, \emph{{{Tangency and duality}}}, {{Proceedings of the 1984
  Vancouver conference in algebraic geometry}}, {{CMS Conference Proceedings}},
  vol.~{6}, {{Amer. Math. Soc.}}, {{1986}}, pp.~{{163--226}}.

\bibitem{KraemerMicrolocalII}
{Kr{\"a}mer, T.}, \emph{{Characteristic cycles and the microlocal geometry of
  the Gauss map II}}, {J.~Reine Angew.~Math.} ({to appear}),
  {\url{arXiv:1807.01929}}.

\bibitem{KrWSchottky}
{Kr{\"a}mer, T. and Weissauer, R.}, \emph{{The Tannaka group of the theta
  divisor on a generic principally polarized abelian variety}}, {Math. Z.}
  \textbf{{281}} ({2015}), {723--745}.

\bibitem{Teissier-Trang}
{L{\^e}, D.~T. and Teissier, B.}, \emph{{Limites d'espaces tangents en
  géométrie analytique}}, {Comment. Math. Helvetici} \textbf{{63}} ({1988}),
  {540--578}.

\bibitem{MacPhersonChern}
{MacPherson, R. D.}, \emph{{Chern classes for singular algebraic varieties}},
  {Ann. of Math.} \textbf{{100}} ({1974}), {423--432}.

\bibitem{MaximIHP}
{Maxim, L.}, \emph{{Intersection homology and perverse sheaves}}, {Graduate
  Texts in Math.}, vol. {281}, {Springer Verlag}, {2019}.

\bibitem{Pirola}
{Pirola, P.}, \emph{{Base number theorem for abelian varieties}}, Math. Ann.
  \textbf{282} (1988), 361--368.

\bibitem{RanSubvarieties}
{Ran, Z.}, \emph{{On subvarieties of abelian varieties}}, {Invent. Math.}
  \textbf{62} (1980), {459--480}.

\bibitem{SabbahSansEclatement}
{Sabbah, C.}, \emph{{Morphismes analytiques stratifiés sans éclatement et
  cycles évanescents}}, {Astérisque} \textbf{{101-102}} ({1983}), {286--319}.

\bibitem{SabbahConormaux}
\bysame, \emph{{Quelques remarques sur la geometrie des espaces conormaux}},
  {Ast\'erisque} \textbf{{130}} ({1985}), {161--192}.

\bibitem{VerdierStratifications}
{Verdier, J.-L.}, \emph{{Stratifications de Whitney et th{\'e}or{\`e}me de
  Bertini-Sard}}, Invent. Math. \textbf{36} (1976), 195--312.

\bibitem{VerraGauss}
{Verra, A.}, \emph{The degree of the {G}auss map for a general {P}rym
  theta-divisor}, J. Algebraic Geom. \textbf{10} (2001), no.~2, 219--246.

\bibitem{WeissauerDegenerate}
{Weissauer, R.}, \emph{{On subvarieties of abelian varieties with degenerate
  Gauss mapping}}, {\url{arXiv:1110.0095}}.

\end{thebibliography}

\end{document}